\documentclass{amsart}
\usepackage{amssymb,amsmath,latexsym,amsthm}
\usepackage[english]{babel}

\usepackage{amsmath,amssymb,amsbsy,amsfonts,amsthm,latexsym,
                       amsopn,amstext,amsxtra,euscript,amscd}
\usepackage{hyperref}
\usepackage{array}
\usepackage{ifthen}
\usepackage{url}

\newtheorem{thm}{Theorem}

\newtheorem{lem}[thm]{Lemma}
\newtheorem{pro}[thm]{Proposition}

\newtheorem{rmq}[thm]{Remark}

\begin{document}

\title[On some ternary Diophantine equations of Signature $(p,p,k)$.]{On some ternary Diophantine equations of Signature $(p,p,k)$.}

\author[Armand Noubissie]{Armand Noubissie}
\address{Institut de Math\'ematiques et de Sciences Physiques. Dangbo, B\'enin}
\email{armand.noubissie@imsp-uac.org}  

\author[Alain Togb\'e]{Alain Togb\'e}
\address{Department of Mathematics, Statistics and Computer Science\\
Purdue University Northwest\\
1401 S, U.S. 421, Westville IN 46391 USA}
\email{atogbe@pnw.edu}



\date{\today}

\maketitle

\begin{abstract}
In this paper, we summarize the work on ternary Diophantine equation of the form $Ax^n+By^n=cz^m$, where $m \in \{2,3,n\} $,  $n\geqslant 7 $ is a prime. Moreover, we completely solve some particular cases ($A=5^{\alpha}, ~B=64,~ c=3, ~m=2; \quad A=2^{\alpha},~ B=27, ~c \in \{7,13\}, ~m=3$). 
\end{abstract}

\section{Introduction}\label{sec1}
Let $p\geq 3$ be a prime number and  Fermat's equation 
\begin{equation}\label{eq:1}
a^p+b^p+c^p=0.
\end{equation}  One knows by Wiles (See \cite{W}) that this equation has no nonzero integral solution. Indeed, suppose that equation \eqref{eq:1} has a nonzero integral solution $(a,b,c)$. One consider the Frey curve $E$ whose equation is $y^2= x(x-a^p)(x+b^p).$ This curve is semistable (see Serre \cite{Se}, Proposition 6,  Section 4) and so it is modular by Wiles' theorem (see \cite{W}, Theorem 0.4). According to Ribet's theorem (see \cite{R}, Theorem 1.1), the Galois representation attached to Frey curve $E$ arise from a newform of weight $2$ and of level $2$, which is impossible because the genus of $X_0(2)$ is $0$. Therefore, no triple $(a,b,c)$ with the hypothesized properties.\\

This method is called the modular method. This method due to Serre, Frey, Ribet and Wiles consists in attaching a putative solution of a Diophantine equation to an elliptic curve $E$ (known as a Frey curve) and to study the Galois representation associated to $E$ via modularity and lowering.  Hence, this Galois representation arises to a newform of weight $2$ and small level and to conclude that such newform doesn't exist.  Notice that the modular method can be adapted to other Diophantine equations of the form
\begin{equation}\label{eq:2}
Ax^p + By^q= Cz^r,
\end{equation} 
for $p,q,r$ positive integers. We will refer to the triple $(p,q,r)$ as the signature  of the corresponding equation. In the case  $(p,p,p) $,  the work of Serre (see \cite{Se}, Theorem 2,  Section 4.3) combined with Ribet's theorem (see \cite{R}, Theorem 1.1) and Wiles' theorem (See \cite{W}, Theorem 0.4), provides the information for 
$$ABC \in \{1,2,3,5,7,11,13,17,19,23,53,59\}.$$ 
Kraus \cite{K1} gave the information for $ABC=15$; Darmon and Merel \cite{DM} for $ABC=2$ and Ribet \cite{R1} for $ABC= 2^{\alpha}$, with $\alpha > 1$. In 1997, Kraus \cite{K2} provided a sufficient criteria for $(A,B,C)$ to guarantee that such an equation \eqref{eq:2} with signature $(p,p,p)$ is insoluble in coprime nonzero integers $(x,y,z)$. Darmon and Granville \cite{DG} showed that  equation \eqref{eq:2} has only finitely many solutions in coprime integers $(x,y,z)$, for $A,B,C,p,q,r$ fixed positive integers such that $p^{-1}+ q^{-1}+r^{-1}< 1$. In case of the signatures $(p,p,2)$ and $(p,p,3)$, a work of Darmon and Merel \cite{DM} provided a comprehensive analysis with $ABC=1$ and Ivorra \cite{I} for $ABC=2$ (case $(p,p,2)$). In 2003, Bennett and Skinner \cite{BS} applied the techniques of Darmon and Merel to the equation of signature $(p,p,2)$ for arbitrary coefficients $A,B,C.$   In 2004, Bennett, Vatsal and Yazdani \cite{BVY} applied the techniques of Bennett and Skinner \cite{BS} in the case of signature $(p,p,3)$.

In this paper, we will provide recipes for solving equation \eqref{eq:2} under some very special conditions in cases $(p,p,2)$ and $(p,p,3)$. Our main theorems are the following.

\begin{thm} \label{thm1}
	Suppose that $n\geqslant 7$ is a prime number. Then the equation
\begin{equation}\label{eq:3}
	5^{\alpha}x^n + 64y^n= 3z^2
\end{equation} has no solution in nonzero coprime integers $(x,y,z)$ with $xy \equiv 1 \pmod 2$ and a positive integer $\alpha$.  
\end{thm}

\begin{thm} \label{thm2}
	Suppose that $n\geqslant 11$  is a prime number. Then the equation 
	\begin{equation}\label{eq:4}
	2^{\alpha}x^n + 27y^n= 7z^3
    \end{equation} has no solution in nonzero coprime integers $(x,y,z)$, where $\alpha$ is a positive integer. 
\end{thm}

\begin{thm} \label{thm3}
Suppose that $n\geqslant 11,$ $n \neq 13$ is a prime number. Then the equation 
\begin{equation}\label{eq:5}
2^{\alpha}x^n + 27y^n= 13z^3
\end{equation} has no solution in nonzero coprime integers $(x,y,z)$, where $\alpha$ is a positive integer. 
\end{thm}

\begin{rmq}
	Notice that these two last theorems are also true if we replace respectively $7$ and $13$ by $7^{\beta}$ and  $13^{\beta}$, for any positive integer
	$\beta$.
\end{rmq}

\section{preliminaries}

We will always assume that $x,y,z,A,B,C$ are nonzero integers such that $Ax,By,Cz$ are coprime, $xy \neq \pm 1$ and satisfying 
$$Ax^n+By^n=cz^m,$$ 
where $n\geq 7$ is a prime in case $m=2$ or $n\geq 11$ is a prime in case $m=3$.

$\bullet$ {\bf Signature $(p,p,2)$.}
Without loss of generality, we will assume  that $Ax$ is odd, $C$ squarefree and we are in one of the following situations:
\begin{enumerate}
\item $abABC \equiv	1 \pmod 2$ and $b \equiv -BC \pmod 4$,
\item $ab \equiv	1 \pmod 2$ and either $\mbox{ord}_2(B)=1$ or $\mbox{ord}_2(C)=1$,
\item $ab \equiv	1 \pmod 2$, $\mbox{ord}_2(B)=2$ and $c \equiv -bB/4 \pmod 4$,
\item $ab \equiv 1 \pmod 2$, and either $\mbox{ord}_2(B)\in \{3,4,5\}$ and $c \equiv C \pmod 4$,
\item $\mbox{ord}_2(Bb^n)\geq 6$ and $c \equiv C \pmod 4.$
\end{enumerate}  
Now, we will consider the three elliptic curves $E_i(x,y,z)$ with $i \in \{1,2,3\}$ given in \cite{BS} by: 

$\star$ in case $(1)$ and $(2)$, we will consider 
$$(E_1): \quad Y^2 = X^3+2cCX^2+BCb^nX,$$  
$\star$ in case $(3)$ and $(4)$, 
$$(E_2): \quad Y^2 = X^3+cCX^2+\frac{BCb^n}{4}X$$ 
$\star$ and in case $(5)$, 
$$(E_3): \quad Y^2 + XY= X^3+\frac{cC-1}{4}X^2+\frac{BCb^n}{64}X.$$ 
They are Elliptic curves defined over $\mathbb{Q}$. Bennett and Skinner \cite{BS} used Tate's algorithm to give the value of the conductor associated to $E_i$ and summarized in the following lemma.
\begin{lem} \label{lem1}(Lemma 2.1 \cite{BS} )
The conductor $N(E)$ of the curve $E=E_i(x,y,z)$ is given by 
$$N(E)= 2^{\alpha}C^2\prod_{p \mid xyAB} p,$$
where
\begin{equation*}
\alpha = \left\{
\begin{aligned}
5 & \quad \mbox{if} \quad i=1, \quad  ABCxy \equiv 1 \pmod 2\\
6& \quad \mbox{if} \quad i=1, \quad  \mbox{ord}_2(B)=1 \quad \mbox{or} \quad \mbox{ord}_2(C)=1\\
1 &\quad \mbox{if} \quad i=2, \quad \mbox{ord}_2(B)=2 \quad \mbox{and} \quad y \equiv -BC/4 \pmod 4\\
2& \quad \mbox{if} \quad i=2, \quad  \mbox{ord}_2(B)=1 \quad \mbox{and} \quad y \equiv BC/4 \pmod 4\\
4& \quad \mbox{if} \quad i=2, \quad \mbox{ord}_2(B)=3 \\
2& \quad \mbox{if} \quad i=2, \quad \mbox{ord}_2(B)\in\{4,5\} \\
-1 &\quad \mbox{if} \quad i=3,\quad  \mbox{ord}_2(By^n)=6\\
0& \quad \mbox{if} \quad i=3, \quad \mbox{ord}_2(B)\geq 7.
\end{aligned}
\right.
\end{equation*}
\end{lem}
To the elliptic curve $E$, we associate a Galois representation  
$$\begin{array}{ccccc}
\rho_n^E : & Gal(\overline{\mathbb{Q}}\mid \mathbb{Q}) & \to & Gl_2(\mathbb{F}_n). \\
\end{array}$$
This is just the representation of $Gal(\overline{\mathbb{Q}}/\mathbb{Q})$ on the $n$-torsion points $E[n]$ of  the elliptic curve $E$, having fixed once and for all an identification of $E[n]$ with $\mathbb{F}_n^2$.  Bennett and Skinner \cite{BS} combined the work of Mazur \cite{M} and Kubert \cite{Ku} to show that if $xy \neq \pm 1$, then $\rho_n^E$ is absolutely irreducible.  So we can associate to the Galois representation  $\rho_n^E$  a number $N_n^E$ called the Artin conductor of  $\rho_n^E$ as it is defined in \cite{S}. If $n \nmid ABC,$ they gave an explicite value of this conductor  $N_n^E$ in the following lemma.
 
\begin{lem}\label{lem3}(Lemma 3.2,  \cite{BS}) We have
$$N_n^E = 2^{\beta}\prod_{p \mid C,~p \neq n} p^2 \prod_{q \mid AB,~q \neq n} q,$$
where
\begin{equation*}
\beta = \left\{
\begin{aligned}
1 & \quad \mbox{if} \quad ab\equiv 0 \pmod 2 \mbox{ and } AB\equiv 1 \pmod 2\\
\alpha& \quad \mbox{otherwise,}
\end{aligned}
\right.
\end{equation*}
for $\alpha$ defined above.
\end{lem}

Applying work of Breuil, Conrad, Diamond and Taylor \cite{BCDT} and a consequence of the work of Ribet (Theorem 6.4 \cite{D}), Bennett and Skinner obtained the following result.

\begin{lem}\label{lem2}({Lemma 3.3, \cite{BS}})
Suppose $xy \neq \pm 1.$ Put 
\begin{equation*}
N_n(E) = \left\{
\begin{aligned}
N_n^E & \quad \mbox{if} \quad n \nmid ABC \\
nN_n^E& \quad \mbox{if} \quad n \mid AB \\
n^2N_n^E &\quad \mbox{if} \quad n \mid C.
\end{aligned}
\right.
\end{equation*}
Then the Galois representation $\rho_n^E$ arises from a cuspidal newform of weight $2$ and level $N_n(E)$ and trivial Nebentypus character.
\end{lem}

$\bullet$ Signature $(p,p,3)$. We assume without loss of generality that $Ax^n \ncong 0 \pmod 3$ and $By^n \ncong 2 \pmod 3$. Furthermore, suppose that $C$ is cubefree, $A$ and $B$ are $n^{th}$-power free. Let us consider the elliptic curve $E'$ given by 
$$E': \;Y^2+3CzXY+C^2By^nY= X^3.$$ 
Bennett, Vatsal and Yazdani  \cite{BVY} showed that the conductor of $E'$ noted $N(E')$ is given by
$$N(E')~=~\epsilon_3  \prod_{p \mid C,~p \neq 3} p^2 \prod_{q \mid ABxy,~q \neq 3}q, $$ 
where 
\begin{equation*}
\epsilon_3= \left\{
\begin{aligned}
3^2& \quad \mbox{if}\quad  9\mid (2+C^2By^n-3Cz)\\
3^3 &\quad \mbox{if} \quad  3\parallel (2+C^2By^n-3Cz) \\
3^4& \quad \mbox{if} \quad  \mbox{ord}_2(By^n)=1 \\
3^3& \quad \mbox{if} \quad  \mbox{ord}_2(By^n)=2\\
1& \quad \mbox{if} \quad \mbox{ord}_3(B)=3 \\
3 &\quad \mbox{if} \quad  \mbox{ord}_3(By^n)> 3\\
3^5& \quad \mbox{if} \quad  3\mid C.
\end{aligned}
\right.
\end{equation*}
The Artin conductor $N_n^{E'}$ associated to the Galois representation $\rho_n^{E'}$ is given by:
$$N_n^{E'}~=~\epsilon_3'  \prod_{p \mid C,~p \neq 3} p^2 \prod_{q \mid AB,~q \neq 3}q, $$ 
where 
\begin{equation*}
\epsilon_3'= \left\{
\begin{aligned}
3^2& \quad \mbox{if}\quad  9\mid (2+C^2By^n-3Cz)\\
3^3 &\quad \mbox{if} \quad  3\parallel (2+C^2By^n-3Cz) \\
3^4& \quad \mbox{if} \quad  \mbox{ord}_2(By^n)=1 \\
3^3& \quad \mbox{if} \quad  \mbox{ord}_2(By^n)=2\\
1& \quad \mbox{if} \quad \mbox{ord}_3(B)=3 \\
3 &\quad \mbox{if} \quad  \mbox{ord}_3(By^n)> 3 \quad and \quad  \mbox{ord}_3(B)\neq 3\\
3^5& \quad \mbox{if} \quad  3\mid C.
\end{aligned}
\right.
\end{equation*}
Suppose that $n\nmid ABC$, then the following lemma is 

\begin{lem}\label{lem3}(Lemma 3.4, \cite{BVY})
	If $xy \neq \pm 1$ and $n\geq 11$ is a prime, then the Galois representation $\rho_n^{E'}$ arises from a cuspidal newform of weight $2$ and level $N_n^{E'}$ and trivial Nebentypus character unless $E'$ corresponds to one of the equations
$$1 \cdot 2^5 + 27 \cdot (-1)^5 = 5 \cdot 1^3 \mbox{ or } 1 \cdot 2^7 + 3 \cdot (-1)^7 = 1 \cdot 5^3.$$
\end{lem}

\section{Eliminating Newforms.}
We will use different methods to eliminate the possibility of certain newforms of level $N_n(E)$, $N_n^{E'}$ giving a rise to the Galois representations $\rho_n^{E}$ and $\rho_n^{E'}$ respectively. See the next propositions. These propositions combine some results in \cite{BS} and \cite{BVY}.
\begin{pro}\label{pro1}(Lemma 4.3, \cite{BS} and Proposition 4.2, \cite{BVY})
	Suppose that $a,b,c,A,B$ and $C$ are nonzero integers with $Aa,Bb,Cc$ coprime, $ab \neq \pm 1$, satisfying 
	$$Aa^n+Bb^n=Cc^m,$$
	with $n\geq 7$ (for $m=2$) or $n\geq 11$ (for $m=3$), where in each case $n$ is a prime. Then there exists a cuspidal newform $f = \sum_{r=1}^{\infty} c_rq^r$ of weight $2$, trivial Nebentypus character and level $N_n(E)$, $N_n^{E'}$ respectively in cases $m=2$ and $m=3$. Moreover, if $K_f$ is the field of definition of Fourier coefficient $c_r$ of $f$ and suppose that $p$ is a prime that is coprime to $nN$, then 
	$$Norm_{K_f\mid \mathbb{Q}}(c_p-a_p) \equiv 0 \pmod n,$$ 
where $a_p= \pm (p+1)$ or $a_p \in S_{p,m}$ with 
	$$S_{p,2}= \{x ~/ \quad \mid x \mid < 2\sqrt{p}, ~~ x \equiv 0 \pmod 2 \}$$ 
and 
	$$S_{p,3}= \{x ~/ \quad  \mid x \mid < 2\sqrt{p}, ~~ x \equiv p+1 \pmod 3 \}.$$
\end{pro}

\begin{pro}\label{pro2}(Proposition 4.4, \cite{BS} and Proposition 4.4, \cite{BVY})
	Suppose that  $n\geq 7$ (for $m=2$) or $n\geq 11$ (for $m=3$) where, in each case, $n$ is a prime. Suppose also that $E"$ is another elliptic curve defined over $\mathbb{Q}$ such that 
	$$ \rho_p^{E} \cong \rho_p^{E"}  \quad \mbox{and} \quad \rho_p^{E'} \cong \rho_p^{E"}  $$ 
respectively. Then the denominator of $j-$invariant $j(E")$ is not divisible by any odd prime $p\neq n$ dividing $C.$
\end{pro}
Notice that Bennett \cite{B} used these two propositions to show that for an odd $n\geq 3$, the Thue equation 
$$x^n-3y^n=2$$ 
has only one integral solution $(x, y)=(-1,-1)$. 

Now, we use the technique developed by the authors in \cite{NT} for solving an exponential Diophantine equation to completely prove Theorem 1. 

\begin{lem}\label{lem4}( Corollary 6.3.15, \cite{C})
	Let $p$ be a positive or negative prime number with $p \neq 2$.   Then the general integral solution of the equation 
	$$x^2+py^2=z^2$$ 
	with $x$ and $y$ coprime is given by one of the following two disjoint parameterizations:\\
	\begin{enumerate}
		\item $x = \pm (s^2-pt^2), ~~y=2st,~~z= \pm (s^2+pt^2)$, where $s$ and $t$ are coprime of the opposite parity such that $p\nmid s.$\\
		\item $x = \pm (((p-1)/2)(s^2+t^2)+ (p+1)st), ~~y=s^2-t^2,~~z= \pm (((p+1)/2)(s^2+t^2)+ (p-1)st)$, where $s$ and $t$ are coprime of the opposite parity such that $s$ not congruent to $t$ modulo $p$.
	\end{enumerate}
\end{lem}
		  
We will prove the next proposition.
\begin{pro} \label{pro3}
	Let $b \in \mathbb{N}$ such that $(b,2)=1$.  Then the Diophantine Equation $$2^{2x}-b^y= \pm 3z^2$$ has no solution in positive integers $(x,y,z)$ with $y$ even integer and $x> 1$. 
\end{pro}
\begin{proof}
We suppose that this equation has a solution in positive integers $(x,y,z)$ with $y$ even integer and $x> 1$. Put $y=2m$, with $m$ be a positive integer. Then,  we get  $$(2^x)^2 \pm 3 z^2= (b^m)^2.$$ As $(2,b)=1,$ it follows that $(2^x,z)=1.$ So by Lemma \ref{lem4}, we have the following two possibilities:\\
\begin{itemize}
	\item $2^x = \pm (s^2 \pm 3t^2), ~~z=2st,~~b^m = \pm (s^2 \mp 3t^2)$, where $s$ and $t$ are coprime of the opposite parity such that $p\nmid s.$\\
	
	\item $2^x = \pm (((\pm 3-1)/2)(s^2+t^2)+ (\pm 3+1)st), ~~z=s^2-t^2,~~b^m= \pm (((\pm 3+1)/2)(s^2+t^2)+ (\pm 3-1)st)$, where $s$ and $t$ are coprime of the opposite parity such that $s$ not congruent to $t$ modulo $3$.
\end{itemize}
The first case is not possible because $z$ is odd.  So we have 
$$2^x = \pm (((\pm 3-1)/2)(s^2+t^2)+ (\pm 3+1)st).$$ 
Thus, we obtain that $2 \parallel \pm (((\pm 3-1)/2)(s^2+t^2)+ (\pm 3+1)st)$ or $2 \nmid 2^x$. This implies $ 2 \parallel 2^x$ or $2 \nmid 2^x$, which is impossible because $4\mid 2^x.$ So our proposition is proved.
\end{proof}

\section{Proof of Theorem \ref{thm1}}\label{sec5}

Suppose that  equation \eqref{eq:3} has a solution in nonzero coprime integers $(x,y,z)$, with $xy \equiv 1 \pmod 2$ and a positive integer $\alpha$.
\begin{itemize}
	\item Case $xy \neq \pm 1.$ The elliptic curve that we consider is just $E=E_3(x,y,z)$ and $N_n(E)= 45$. By Lemma \ref{lem2}, it follows that there exists  a newform of weight $2$, level $45$ and trivial Nebentypus character. This newform corresponding to the curve over $\mathbb{Q}$ is $45A$ in Cremona's table \cite{C}. Since this curve has a $j$-invariant with denominator $15$, then we may apply Proposition \ref{pro2} to conclude as desired.
	\item Case $xy = \pm 1.$ If $xy = -1$, then equation \eqref{eq:3} becomes $64-5^{\alpha}= \pm 3z^2.$ By Proposition \ref{pro3}, it follows that $\alpha$ is odd. So we have $64-5^{\alpha} \equiv 5 \pmod 6.$ Moreover, $\pm 3z^2 \equiv 0, \pm 3 \pmod 6$ and so $5 \equiv 0, \pm 3 \pmod 6$, which is a contradiction.  If  $xy = 1$, then equation \eqref{eq:3} becomes $64+5^{\alpha}= \pm 3z^2.$ We have $64+5^{\alpha} \equiv -1 \pmod 5$ and $\pm 3z^2 \equiv 0, \pm 3 \pmod 5.$ So we obtain $-1 \equiv 0, \pm 3 \pmod 5$, which is a contradiction. So the proof of Theorem \ref{thm1} is complete. 
\end{itemize}

\section{Proof of Theorem \ref{thm2}}\label{sec6}
Suppose that equation \eqref{eq:4} has solution in nonzero coprime integers $(x,y,z)$ with $\alpha$ positive integer when $n\geq 11.$ 
\begin{itemize}
\item Case $xy \neq \pm 1$. The elliptic curve that we consider is $E'$ and $N_n(E)=98$. Lemma \ref{lem3} tells us $\rho_n^{E'}$ arised from the newform of weight $2$, level $98$ and trivial character. At this level, there are $2$ classes of newforms to consider by Stein's database \cite{S}, denoted by $98,1;~98,2.$ For $98,1$; we have $c_3=2.$ So we deduce a contradiction to Proposition \ref{pro1} by a consideration  of a single Fourier coefficient. For $98,2$; we have $c_3=\theta$, where  $\theta$ is a root of polynomial $x^2-2.$ So, by Proposition \ref{pro1}, the prime $n$ must divide 
$$\mid Norm_{K_f\mid \mathbb{Q}}(c_p-a_p)\mid  \in \{1,14\},$$ 
where $K_f= \mathbb{Q}(\sqrt{2}) $. We get a contradiction as $n\geq 11.$
\item Case $xy = \pm 1$. In this case, equation \eqref{eq:4} becomes 
$$(\pm 7z)^3= 49 \times 2^{\alpha} \pm 1323.$$
If $\alpha$ is even, then the above equation becomes 
$Y^2=X^3 \pm 1323,$ where $X=\pm 7z$ and $Y=7\times 2^m,$ with $\alpha =2m, ~~m\geq1.$ Using Magma, we get that this Mordell equation has integral solutions 
$$(X,Y) \in \{( -3, \pm 36) \}.$$
 This is impossible. If $\alpha$ is odd, $\alpha= 2m+1,$ with $m \geq 1$, then equation \eqref{eq:4} becomes $Y^2=X^3 \pm 10584,$ where $X=\pm 14z$ and $Y=28 \times 2^m.$ Using Magma, we get that this Mordell equation has integral solutions 
$(X,Y) \in \{(22, \pm 8), (25, \pm 71),(42, \pm 252), (105, \pm 1071), (294, \pm 5040), (394, \pm 7820) \}.$\\	 
This gives no solution to the initial equation and completes the proof of Theorem \ref{thm2}.
\end{itemize}

\section{Proof of Theorem \ref{thm3}}\label{sec7}
Assume that equation \eqref{eq:5} has a nonzero coprime integer solution $(x,y,z)$ such that $n\geqslant 11,$ $n \neq 13$ is a prime number.
\begin{itemize}
	\item Case $xy \neq \pm 1.$ The elliptic curve is $E'$ and $N_n(E')=338$. So Lemma \ref{lem3} tells us $\rho_n^{E'}$ arised from the newform of weight $2$, level $338$ and trivial character. At this level, there are $8$ classes of newforms by Stein's table \cite{S}, denoted by $338,1;~338,2; ~338,3; ~338,4;  ~338,5~ 338,6;$ $~ 338,7; ~ 338,8.$ For $338,1;~338,2;~338,4;$ and we have $c_3= 0, -3.$ So by Proposition \ref{pro1}, we get a contradiction. For $338,3;~338,5;~338,6$; one can see that $c_7= \pm 3, 1$, which contradicts Proposition \ref{pro1}. For newforms $338,7;~338,8;$ we have $c_3= \theta$ and $c_5= \pm (2\theta^2 - 12)$, where $\theta$ is a real root of polynomial $x^3-3x^2-4x+13.$ We compute $ Norm_{K_f \mid \mathbb{Q}}(c_p-a_p) $ and find that 
	$$\mid  Norm_{K_f \mid \mathbb{Q}}(c_p-a_p) \mid \in \{7,13,83\}.$$ 
We must obtain $n=83$ as $n\geq 11$ and $n \neq 13.$ Therefore, $\theta \equiv 4 \pmod \wp $ for $\wp$ a prime lying above $83$ (Notice that this prime is $\wp = \theta - 4$ according to Alaca and William \cite{A}). Since $c_5 = \pm (2\theta^2 - 12),$ then for these newforms it follows that $a_5 \equiv \pm 20 \pmod {83}$. This means that 
	$$ 0, \pm 3, \pm 6 \equiv \pm 20 \pmod {83},$$ 
which is impossible.
	\item Case $xy = \pm 1$. In this case, equation \eqref{eq:5} becomes 
	$$(\pm 13z)^3= 169 \times 2^{\alpha} \pm 4563.$$
	If $\alpha$ is even, then the above equation is 
	$Y^2=X^3 \pm 4563$, where $X=\pm 13z$ and $Y=13 \times 2^m,$ with $\alpha =2m, ~~m\geq1.$ Using Magma, we get that this Mordell equation has the integral solutions 
	$$(X,Y) \in \{(39, \pm 234) \}.$$ 
	This doen't give any solution to the initial equation. If $\alpha$ is odd, $\alpha= 2m+1,$ with $m \geq 1$, then equation \eqref{eq:5} becomes $Y^2=X^3 \pm 36504,$ where $X=\pm 26z$ and $Y=52 \times 2^m.$ Using Magma, we obtain that this Mordell equation has the integral solutions 
		$$(X,Y) \in \{(30, \pm 252) \}.$$ 
	We deduce no solution and this completes the proof of Theorem \ref{thm3}.
\end{itemize}

\bibliographystyle{plain}

\end{document}